\documentclass[headsepline, auto-pst-pdf, a4paper ,openany, 11pt]{article}
\usepackage{geometry}                
\usepackage[parfill]{parskip}    
\usepackage{a4wide}
\usepackage{graphicx,epsfig,color,caption} 
\usepackage{amsxtra}
\usepackage{amsthm, amssymb,bbm,mathtools}
\usepackage{epstopdf}
\DeclareGraphicsExtensions{.pdf,.png,.jpg}

\newtheorem{theorem}{Theorem}[section]
\newtheorem{lemma}[theorem]{Lemma}

\newtheorem{conjecture}[theorem]{Conjecture}
\newtheorem{proposition}[theorem]{Proposition}

\newtheorem{question}[theorem]{Question}

\newtheorem{definition}[theorem]{Definition}

\newcommand{\R}{\ensuremath{\mathbbm{R}}}
\newcommand{\cA}{\ensuremath{\mathcal{A}}}

\newcommand{\cE}{\ensuremath{\mathcal{E}}}
\newcommand{\cL}{\ensuremath{\mathcal{L}}}
\newcommand{\cM}{\ensuremath{\mathcal{M}}}
\newcommand{\cH}{\ensuremath{\mathcal{H}}}

\newcommand{\PP}{\mathbbm{P}}

\newcommand{\F}{\ensuremath{\mathbbm{F}}}
\newcommand{\I}{\ensuremath{\mathcal{I}}}
\newcommand{\bv}{\ensuremath{\mathbf{v}}}
\newcommand{\bs}{\ensuremath{\mathbf{s}}}
\newcommand{\cP}{\ensuremath{\mathcal{P}}}

\newcommand{\size}[1]{\left| #1 \right|}		

\renewcommand{\geq}{\geqslant}
\renewcommand{\leq}{\leqslant}

\newcommand*\widebar[1]{%
  \hbox{%
    \vbox{%
      \hrule height 0.5pt 
      \kern0.5ex
      \hbox{%
        \kern-0.1em
        \ensuremath{#1}%
        \kern-0.1em
      }%
    }%
  }%
}

\title{On the minimum degree of minimal Ramsey graphs for multiple colours}
\author{
Jacob Fox\thanks{
    Department of Mathematics,
   Stanford University,
    Stanford, CA 94305.
    Email: {\tt fox@math.mit.edu}.
    Research supported by a Packard Fellowship, by NSF Career Award DMS-1352121, and by a Sloan Foundation Fellowship.}  \and 
Andrey Grinshpun\thanks{
    Department of Mathematics,
    Massachusetts Institute of Technology,
    Cambridge, MA 02139-4307.
    Email: {\tt agrinshp@math.mit.edu}.
    Research supported by a National Physical Science Consortium Fellowship.} \and
Anita Liebenau\thanks{
	Department of Computer Science, University of Warwick, Coventry CV4 7AL, UK. This research 
	was done when the author was affiliated with the Institute of Mathematics, Freie Universit\"at Berlin, 14195 Berlin, 		Germany. Email: {\tt a.liebenau@warwick.ac.uk}. The author was supported by the Berlin Mathematical School. 
	The author would like to thank the Department of Mathematics,
    	Massachusetts Institute of Technology, Cambridge, MA 02139-4307 for its hospitality where this work was 			partially carried out.} 
	\and
Yury Person\thanks{Institute of Mathematics, Goethe-Universit\"at, 60325 Frankfurt am Main, Germany.  Email: {\tt person@math.uni-frankfurt.de}} \and
Tibor Szab\'o\thanks{
Institute of Mathematics, Freie Universit\"at Berlin, 14195 Berlin, Germany. Email: {\tt szabo@math.fu-berlin.de} Research partially supported by DFG within the Research Training Group {\em Methods for Discrete Structures} and by project SZ 261/1-1.}
}

\begin{document}

\maketitle

\begin{abstract}
A graph 
$G$ is $r$-Ramsey for a graph
$H$, denoted by $G\rightarrow (H)_r$, if every $r$-colouring of the edges 
of $G$ contains a monochromatic copy of $H$. 
The graph $G$ is called $r$-Ramsey-minimal for $H$ if it is $r$-Ramsey 
for $H$ but no proper subgraph of $G$ possesses this property.
Let $s_r(H)$ denote the smallest minimum degree of $G$ over all graphs $G$ that are $r$-Ramsey-minimal for $H$. The study of the parameter $s_2$ was initiated by Burr, 
Erd\H{o}s, and Lov\'{a}sz in 1976 when they showed that for the clique 
$s_2(K_k)=(k-1)^2$. 
In this paper, we study the dependency of $s_r(K_k)$ on $r$ and 
show that, under the condition that $k$ is constant, 
$s_r(K_k) = r^2\cdot \text{polylog} \ r$. We also give an upper bound 
on $s_r(K_k)$ which is polynomial in both $r$ and $k$, 
and we determine $s_r(K_3)$ up to a factor of $\log r$.  
\end{abstract}


\section{Introduction}

A graph $G$ is $r$-Ramsey for a graph
$H$, denoted by $G\rightarrow (H)_r$, if every $r$-colouring of the edges 
of $G$ contains a monochromatic copy of $H$. 
The fact that, for any number of colours $r$ and every graph $H$, there exists 
a graph $G$ such that $G\rightarrow (H)_r$ is a consequence of Ramsey's 
theorem \cite{R30}. 
Many interesting questions arise when we consider graphs $G$ which are 
minimal with respect to $G\rightarrow (H)_r$.
A graph $G$ is {\em $r$-Ramsey-minimal for $H$} 
(or {\em $r$-minimal for $H$})
if $G \rightarrow (H)_r$, but  $G' \nrightarrow (H)_r$ for any proper subgraph $G' \varsubsetneq G$. 
Let $\cM_r(H)$ denote the family of all graphs $G$ 
that are $r$-Ramsey-minimal with respect to $H$.  
Ramsey's theorem implies that $\cM_r(H)$ is non-empty for all integers $r$ and all finite graphs $H$. 
However, for general $H$, it is still widely open to classify the graphs in $\cM_r(H)$, or even to prove that these graphs have certain properties.

Of particular interest is $H=K_k$, the complete graph on $k$ vertices, and a fundamental problem is to estimate various parameters of graphs $G \in \cM_r(K_k)$, 
that is, of $r$-Ramsey-minimal graphs for the clique on $k$ vertices. The best-studied such parameter is the Ramsey number $R_r(H)$, the smallest number of vertices of any graph in $\cM_r(H)$. Estimating $R_r(K_k)$, or even $R_2(K_k)$, is one of the main open problems in Ramsey theory. 
Classical results of Erd\H{o}s \cite{E47} and Erd\H{o}s and Szekeres \cite{ES35} 
show that $2^{k/2}\leq R_2(k) \leq 2^{2k}$. 
While there have been several improvements on these bounds (see for example~\cite{C09}
and~\cite{Spe75}), the constant factors in the above exponents remain the same.
For multiple colours, the gap between the bounds is larger. Even for the triangle $K_3$,  
the best known upper bound on the $r$-colour Ramsey  number $R_r(K_3)$ 
is of order $2^{O(r \ln r)}$ \cite{w1997}, 
whereas, from the other side, $R_r(K_3)\geq 2^{\Omega(r)}$ is the best known lower 
bound (see \cite{xxer2004} for the best known constant).

Other properties of $\cM_r(K_k)$ have also been studied: R\"odl and Siggers showed in \cite{rs2008} that, for all 
$k\geq 3$ and $r\geq 2$, there exists a constant $c=c(r,k)>0$ 
such that, for $n$ large enough, there are at least $2^{cn^2}$ non-isomorphic 
graphs $G$ on at most $n$ vertices that are $r$-Ramsey-minimal for the clique $K_k$. 
In particular, $\cM_r(K_k)$ is infinite. Another well-studied parameter is the size Ramsey number $\hat R_r(H)$ of a graph $H$, which is the minimum number of edges of a graph in $\cM_r(K_k)$.  

Interestingly, some extremal parameters of graphs in $\cM_r(K_k)$ 
could be determined exactly when the number of colours is two. 
In this paper, we consider the minimal minimum degree of 
$r$-Ramsey-minimal graphs $s_r(H)$, defined by
\[ s_r(H) := \min_{G \in \cM_r(H)} \delta(G) \]
where $\delta (G)$ denotes the minimum degree of $G$.

It is rather simple to see that, for any graph $H$,
\begin{align}\label{eq:simplebounds}
r(\delta(H)-1) < s_r(H) < R_r(H).
\end{align} 
Indeed, for $r=2$, the proof of the lower bound is included in \cite{fl2006}; 
it generalises easily to more colours. We include a similar argument at the beginning of Section \ref{sec:lowerbdds}.
In \cite{bel1976}, Burr, Erd\H{o}s, and Lov\'asz showed that, 
rather surprisingly, the simple upper bound above is far from optimal when 
$r=2$, namely $s_2(K_k) = (k-1)^2$. 

In this paper, we study the behaviour of $s_r(K_k)$ as a function of $r$ and $k$. 
We mainly study $s_r(K_k)$ as a function of $r$ with $k$ fixed.
In particular, we determine $s_r(K_3)$ up to a logarithmic factor.
\begin{theorem}\label{ManyColoursTriangle}
There exist constants $c,C>0$ such that for 
all $r\geq 2$, we have  
$$ c r^2 \ln r \leq s_r(K_3) \leq C r^2 \ln^2 r.$$
\end{theorem}
One can show that $s_r(K_k)\geq s_{r-1}(K_k)$  
(this follows from a stronger statement, cf.~Theorem~\ref{fundamental} and 
Proposition~\ref{cor:LowerBddWithErdoesRogers}). 
However, it is not clear
that $s_r(K_k) \geq s_r(K_{k-1})$.
Therefore, the lower bound on $s_r(K_3)$ does not necessarily imply a similar lower 
bound on $s_r(K_{k})$.
We can in fact only prove a super-quadratic lower bound on
$s_r(K_k)$ that is slightly weaker. 
\begin{theorem}\label{ManyColoursGeneral}
For all $k\geq 4$ there exist constants $c=c(k), C=C(k) >0$ such that, for all $r \geq 3$, 
$$ c\,  r^2\, \frac{\ln r}{\ln \ln r}
 \leq s_r(K_k) 		
	\leq C\, r^2(\ln r)^{8(k-1)^2}.$$
\end{theorem}
The proof of the upper bounds in Theorems~\ref{ManyColoursTriangle} and~\ref{ManyColoursGeneral} 
are of asymptotic nature and require $r$ to be rather large.
Moreover, the exponent of the $(\ln r)$-factor in the latter upper bound depends 
on the size of the clique.  
Therefore, we also prove an upper bound on $s_r(K_k)$ which is 
polynomial both in $r$ and in $k$ and is applicable for small values 
of $r$ and $k$. 
\begin{theorem}\label{ManyColoursGeneralWeak}
For $k\geq 3$, $r\geq 3$, 
$ s_r(K_k) \leq 8 (k-1)^6 r^3.$
\end{theorem}

\textbf{Tools.}
We give an overview of the tools we use to prove bounds on $s_r(K_k)$. The first step will be to reduce finding $s_r(K_k)$ to a simpler problem. 
We call a sequence of pairwise edge-disjoint graphs $G_1,\ldots,G_r$ on the
same vertex set $V$ a {\em colour pattern} on $V$.
For a graph $H$, a colour pattern $G_1,\ldots,G_r$ is called {\em $H$-free}
if none of the $G_i$ contains $H$ as a subgraph.
A graph with coloured vertices and edges is called {\em strongly monochromatic} if 
all its vertices and edges have the same colour.

\begin{definition}
The {\em $r$-colour  $k$-clique packing number}, $P_r(k)$, is the smallest integer $n$ such that there exists a $K_{k+1}$-free colour pattern $G_1,\ldots,G_r$  on an $n$-element
vertex set $V$ 
with the property that any $[r]$-colouring of $V$ contains a strongly 
monochromatic $K_k$. 
\end{definition}
While Burr, Erd\H{o}s, and Lov\'asz~\cite{bel1976} do not explicitly define $P_2(k)$ in their 
proof of $s_2(K_k)=(k-1)^2$, they do essentially show that $s_2(K_{k})=P_2(k-1)$ 
and it is then not hard to see that $P_2(k-1)=(k-1)^2$. 
Here we generalise their result to an arbitrary number $r$ of colours.
\begin{theorem}\label{fundamental}
For all integers $r,k\geq 2$ we have $s_r(K_{k+1})=P_r(k)$.
\end{theorem}
The lower bound $s_r(K_{k+1}) \geq P_r(k)$ is not difficult to derive from the definitions. 
The upper bound $s_2(K_{k+1}) \leq P_2(k)$ follows from a powerful theorem of \cite{bel1976}. We use later generalisations of this theorem by Burr, Ne\v{s}et\v{r}il, and R\"{o}dl \cite{bnr1984} and, recently in 2008, by R\"{o}dl and Siggers \cite{rs2008} to derive $s_r(K_{k+1}) \leq P_r(k)$ for arbitrary $r\geq 2$. 

The problem then becomes to obtain bounds on $P_r(k)$. 
We will see that $P_r(k)$ 
relates closely to the so-called {\em Erd\H os-Rogers function}, 
which was first studied by Erd\H os and Rogers~\cite{er1962} in 1962.
We will be particularly concerned with the special case of the Erd\H{o}s-Rogers function, denoted by $f_{k,k+1}(n)$, which is defined to be the largest integer $\alpha$ so that in any $K_{k+1}$-free graph on $n$ vertices, there must be a vertex-set of size $\alpha$ that contains no $K_k$. 
For our bounds, we will rely heavily on the modern analysis of $f_{k, k+1}$ found in
\cite{dm2012, drr2013, dr2011, s1983}. 
In Section~\ref{sec:lowerbdds}, we will see that essentially
$P_r(k)= \Omega \left(r(f_{k,k+1}(r))^2\right)$, so lower bounds on $f_{k,k+1}$
directly translate to lower bounds on $P_r(k)$. 
In Section~\ref{sec:packings}, we obtain upper bounds on $P_r(k)$ by packing $r$ 
graphs, each giving good upper bounds on $f_{k,k+1}$, into the same vertex set.

\textbf{Organisation.} 
In the next section, we prove that $s_r(K_{k+1}) = P_r(k)$. In Section \ref{sec:lowerbdds}, we prove the lower bounds on $P_r(k)$ in 
Theorem \ref{ManyColoursTriangle} and Theorem \ref{ManyColoursGeneral}. In Section \ref{sec:packings}, we prove the upper bounds in Theorem \ref{ManyColoursTriangle}, Theorem \ref{ManyColoursGeneral}, and Theorem \ref{ManyColoursGeneralWeak}.
We close this paper with some concluding remarks. 

\section{Passing to $P_r(k)$}
In this section we conclude Theorem~\ref{fundamental} from Lemmas~\ref{fundamentalLower} and \ref{fundamentalUpper}.   
\begin{lemma}\label{fundamentalLower}
For all $r,k \geq 1$, we have $s_r(K_{k+1}) \geq P_r(k)$.
\end{lemma}
\begin{proof}
Let $G$ be an $r$-Ramsey-minimal graph for $K_{k+1}$ with a vertex $v$ of degree $s_r(K_{k+1})$. Let $\chi:E(G-v) \rightarrow [r]$ be an $r$-colouring of $G-v$ 
without a monochromatic $K_{k+1}$; such a colouring exists by the minimality of $G$. 
Let $G_1, \ldots , G_r\subseteq G[N(v)]$ be the pairwise edge-disjoint subgraphs
of the $r$ colours within the neighbourhood $N(v)$ of $v$; they form a $K_{k+1}$-free
colour pattern on $N(v)$.  We show that any vertex-colouring of $G[N(v)]$ 
must contain a strongly monochromatic $k$-clique and hence, by the definition of 
$P_r(k)$, the number of vertices $|N(v)|=s_r(K_{k+1})$ must be at least $P_r(k)$.
Indeed, given any vertex-colouring of $N(v)$ we may define an extension of $\chi$
to the edges incident to $v$ by colouring an edge $vu$ with the colour of the vertex 
$u \in N(v)$. 
Since $G$ is $r$-Ramsey for $K_{k+1}$, this extension of $\chi$ contains
a monochromatic $(k+1)$-clique $H$. Moreover, $H$ must contain $v$
(as $\chi$ was free of monochromatic $K_{k+1}$). 
By the definition of the extension of $\chi$, 
the vertices of $H$ in $N(v)$ form a strongly monochromatic $K_k$ in $G[N(v)]$.
\end{proof}

In order to show $s_r(K_{k+1}) \leq P_r(k)$, we first prove a theorem 
that guarantees, for any integer $r \geq 2$ and graph $H$ which is
$3$-connected or a triangle, a fixed colour pattern on a given induced
subgraph of some graph $G$ which is not $r$-Ramsey for
$H$, in {\em any} monochromatic $H$-free $r$-colouring of $G$. 
A similar theorem was proved for $H=K_k$ and for $r=2$ in \cite{bel1976}, 
where they use it to show $s_2(K_{k+1}) \leq P_2(k)$. The tools used to prove this were generalised to any $3$-connected graph $H$ in \cite{bnr1984}, and, more recently, to any number of colours and any graph $H$ which is $3$-connected or a triangle \cite{rs2008}.

\begin{theorem}\label{bel}
Let $H$ be any $3$-connected graph or $H=K_3$ and let $G_1,\ldots,G_r$ be an 
$H$-free colour pattern. Then there is a graph $G$ with an induced copy of the 
edge-disjoint union $G_1 \cup \cdots \cup G_r$ so that $G\nrightarrow (H)_r$ 
and in any monochromatic $H$-free $r$-colouring of $E(G)$  
each $G_i$ is monochromatic and no two distinct $G_i$ and $G_j$ are 
monochromatic of the same colour.
\end{theorem}

\begin{proof}
We use the idea of signal sender graphs which was first introduced 
by Burr, Erd\H{o}s and Lov\'{a}sz \cite{bel1976}. 
Let $r\geq 2$ and $d\geq 0$ be integers and $H$ be a graph. 
A {\em negative (positive) signal sender $S=S^-(r,H, d)$ ($S=S^+(r,H,d)$)}
is a graph $S$  with two distinguished 
edges $e,f \in E(S)$ of distance at least $d$, such that 
\begin{itemize}
\item[$(a)$] $S \nrightarrow (H)_r$, and 
\item[$(b)$] in every $r$-colouring of $E(S)$ without a monochromatic copy of $H$, 
the edges $e$ and $f$ have different (the same) colours.
\end{itemize}
We call $e$ and $f$ the {\em signal edges} of $S$.

Burr, Erd\H{o}s and Lov\'{a}sz \cite{bel1976} 
showed that positive and negative signal senders exist for arbitrary $d$
in the special case when the number of colours is two, and $H$ is a clique on at least
three vertices.
Later, Burr, Ne\v{s}et\v{r}il and R\"{o}dl \cite{bnr1984} 
extended these results to arbitrary $3$-connected $H$.
Finally, R\"{o}dl and Siggers \cite{rs2008} 
constructed positive and negative signal senders 
$S^-(r,H,d)$ and $S^+(r,H,d)$ for any $r\geq 3$, $d\geq 0$ 
as long as $H$ is $3$-connected or $H=K_3$. 


Let $H$ be a graph that is either $3$-connected or $H=K_3$ and let 
$G_1,\ldots,G_r$ be an $H$-free colour pattern on vertex set $V$. 
We construct our graph $G$ using the signal senders of R\"odl and Siggers.
We first take the graph on $V$ which is the edge-disjoint union of the edge sets of the 
graphs $G_i$ and add $r$ isolated edges $e_1,\ldots,e_r$ disjoint from $V$. 
Then for every $i$ and every edge $f\in E(G_i)$ we add a copy of $S^+(r,H,|V(H)|)$, such that
$f$ and $e_i$ are the two signal edges and the sender graph is otherwise 
disjoint from the rest of the construction. Finally,
for every pair of edges $e_i,e_j$, we add a copy of $S^-(r,H, |V(H)|)$, such that
$e_i$ and $e_j$ are the two signal edges and the sender graph is otherwise 
disjoint from the rest of the construction. 

By the properties of positive and negative signal senders, in any $r$-colouring of $G$ without a monochromatic 
 $H$, each $G_i$ must be monochromatic and no two $G_i,G_j$ may be monochromatic in the same colour.

Now we need only to show that there exists an $r$-colouring of $G$ with no 
monochromatic $H$. 
For this, we first colour each $G_i$ with colour $i$. Then, we extend this colouring 
to a colouring of each signal sender so that each signal sender contains no 
monochromatic copy of $H$. This is possible since each positive (negative) signal sender 
has a colouring without a monochromatic copy of $H$ in which the signal edges have 
the same (different) colours. 
 Let us consider a copy of $H$ in $G$. We will see that $H$ is 
 contained either within $G_1 \cup \cdots \cup G_r$ or within one of the signal senders
 and hence it is not monochromatic.  
If this was not the case, then there would be 
a vertex $v_1$ of $H$ that is not in any of the signal edges,
that is, $v_1\in V(S)$ for some signal sender $S$ but not contained in any of the two signal edges of $S$. Since $H$ is not entirely in $S$, there must be
a vertex $v_2\in V(H)\setminus V(S)$. This immediately implies that $H\neq K_3$, since 
$v_1$ and $v_2$ are not adjacent. Since $H$ is $3$-connected there are three internally
disjoint $v_1,v_2$-paths in $H$. These paths can leave $S$ only through one of 
its two signal edges. Hence there is a path of $H$ in $S$ between the two signal edges. This is a contradiction because the distance of the two signal edges in $S$ is at least $|V(H)|$.
%
\end{proof}


Theorem \ref{bel} allows us to finish the proof of Theorem~\ref{fundamental}.
\begin{theorem} \label{fundamentalUpper} 
$s_r(K_{k+1}) \leq P_r(k)$.
\end{theorem}
\begin{proof}
Let a $K_{k+1}$-free colour pattern $G_1,\ldots,G_r$ be given on vertex set 
$V$ with $\size{V} = P_r(k)$, so that any $[r]$-colouring 
of $V$ contains a strongly monochromatic $K_k$. 
Take $G$ as in Theorem \ref{bel} with $H=K_{k+1}$, 
and define $G'$ to be $G$ with a new vertex $v$ which is incident only to $V$. 
We claim that $G'\rightarrow K_{k+1}$, that is for any 
$r$-colouring $\chi$ of $G'$ we find a monochromatic $K_{k+1}$.
If already the restriction of $\chi$ to $V(G)$ contains a monochromatic $K_{k+1}$ 
then  we are done. Otherwise, by 
Theorem~\ref{bel}, 
we have that, after potentially permuting the colours, each subgraph $G_i \subseteq G[V]$ 
is monochromatic in colour $i$. We define a colouring of $V$ by colouring 
$u \in V$ with $\chi(uv)$. 
Then, by the choice of $G_1,\ldots,G_r$, there is a strongly monochromatic clique in $V$.
This clique along with vertex $v$ forms a monochromatic $K_{k+1}$ in the colouring $\chi$. 

So $G' \rightarrow K_{k+1}$. Now observe that any $r$-Ramsey-minimal subgraph of  $G'$
must contain the vertex $v$, since  $G'-v=G$ is not $r$-Ramsey for $K_{k+1}$ by 
Theorem~\ref{bel}.
Hence for the minimum degree of any $r$-Ramsey-minimal subgraph $G''\subseteq G'$ 
we have that 
$s_r(K_{k+1}) \leq \delta (G'') \leq \deg_{G''}(v) \leq \deg_{G'}(v) = P_r(k)$.
\end{proof}

\section{Lower bounds on $P_r(k)$}\label{sec:lowerbdds}

First, we 
prove a simple linear 
lower bound on $P_r(k)$. This simple estimate will later be used to obtain a super-quadratic lower bound.

\begin{lemma}\label{simpleBdd}
For all $r\geq 2$ and $k\geq 3$, we have $P_r(k) > (k-1)r$.
\end{lemma}
\begin{proof}
We will show that for any given colour pattern $G_1,\ldots,G_r$ on vertex set $V$, 
$|V|\leq (k-1)r$,
there is a vertex-colouring of $V$ without a strongly monochromatic $K_k$ and hence,
$P_r(k) > (k-1)r$.
Observe that every vertex $v\in V$ has degree at most $k-2$ in at least one of 
the colour classes, say $G_{i(v)}$. 
Colouring vertex $v$  with colour $i(v)$ ensures that $v$ is not contained 
in any strongly monochromatic $K_k$, as its degree in $G_{i(v)}$ is too low. 
Hence, as promised, 
this vertex-colouring of $V$ produces no strongly monochromatic $K_k$.
\end{proof}

For a graph $F$,  the {\em $k$-independence number} $\alpha_k(F)$ is 
the largest cardinality of a subset $I \subseteq V(F)$ without a $K_{k}$. 
For $k=2$, this is the usual independence number $\alpha(F)$.
Recall that the Erd\H{o}s-Rogers function $f_{k,k+1}(n)$ is defined to be 
the minimum value of $\alpha_k(F)$ over all $K_{k+1}$-free graphs $F$ on $n$ vertices. 

The following proposition provides the recursion for our lower bound.

\begin{proposition}\label{cor:LowerBddWithErdoesRogers}
For all $r,k\geq 2$ we have that $P_r(k)$ satisfies the following inequality: 
$$P_r(k)\geq P_{r-1}(k) + f_{k,k+1}(P_r(k))$$
\end{proposition}
\begin{proof}
Take $G_1,\ldots,G_r$ to be a $K_{k+1}$-free colour pattern  on vertex set 
$V$, $|V|=P_r(k)$, so that any $r$-colouring of the vertices
contains a strongly monochromatic $K_k$.
Let $I\subseteq V$ be a $k$-independent set 
of size $\alpha_k(G_r)$ in the graph $G_r$. 
We claim that the $K_{k+1}$-free colour pattern $G_1, \ldots , G_{r-1}$ restricted
to the vertex set $V\setminus I$ has the property that any 
$[r-1]$-colouring $c: V\setminus I \rightarrow [r-1]$ 
contains a strongly monochromatic $K_k$.
Indeed, the extension of $c$ to $V$ which colours the vertices in $I$ with colour $r$
must contain a strongly monochromatic $K_k$ and this must be inside $V\setminus I$,
since $I$ does not contain $K_k$ at all. 
Hence $\size{V\setminus I} \geq P_{r-1}(k)$ and then, since $G_r$ is a 
$K_{k+1}$-free graph on $P_r(k)$ vertices, we have that
\begin{equation*}
P_r(k) = \size{V\setminus I} + \size{I} 
	\geq P_{r-1}(k) + \alpha_k(G_r)
	\geq P_{r-1}(k) + f_{k,k+1}(P_r(k)). \qedhere
\end{equation*}
%
\end{proof}


Therefore, we are interested in good lower bounds on the Erd\H{o}s-Rogers function 
$f_{k,k+1}(n)$.
It is easy to see that every $K_{k+1}$-free graph $F$ on $n$ vertices contains a $K_k$-free set
of size at least $\left\lfloor\sqrt{n}\right\rfloor$. If  there exists a vertex $v$ of degree 
at least $\left\lfloor\sqrt{n}\right\rfloor$, then $N(v)$ is a $K_k$-free set of
size at least $\left\lfloor\sqrt{n}\right\rfloor$. 
Otherwise, $\Delta(F) \leq \left\lfloor\sqrt{n}\right\rfloor -1 $ and we can use the well-known fact that $\alpha (F) \geq n/(\Delta(F) +1)$ 
(cf. \cite{as2008}) to deduce that $\alpha_k(F)\geq 
\alpha(F)\geq \left\lfloor\sqrt{n}\right\rfloor$. 
Therefore, $f_{k,k+1}(n) \geq \left\lfloor\sqrt{n}\right\rfloor$. 

A result of Shearer \cite{s1983} implies that 
$f_{2,3}(n)\geq \left(1-o(1)\right) \sqrt{(n \ln n)/2}$, which is the best known lower bound 
on $f_{2,3}(n)$. Bollob\'{a}s and Hind \cite{bh1991} proved that  
$f_{3,4}(n)\geq \sqrt{2n}$. This lower bound
was subsequently 
improved by Krivelevich \cite{k1994}. 
Recently, Dudek and Mubayi \cite{dm2012} showed that this result can be strengthened to 
$$ f_{k,k+1}(n)=\Omega \left(\sqrt{\frac{n \log n}{\log \log n}}\right)$$
by using a result of Shearer \cite{s1995}.

{\em Proof of the lower bounds in Theorem~\ref{ManyColoursTriangle}  and \ref{ManyColoursGeneral}.}
Let $k$ be fixed and and for brevity let us write $P_r:=P_r(k)$. 
Let $f_{k,k+1}(n) \geq g(n)\sqrt{n}$ for $n\geq n_0$, where 
$g(n)=g_k(n)$ is a non-decreasing function such that 
$\frac{Cg^2(n-1)}{n} > g^2(n)-g^2(n-1)>0$ for $n\geq n_0$ with some constant $C=C(k)$. 
Note that one can take $g_2(n) = \frac{1}{2}\sqrt{\ln n}$ by \cite{s1983} and for $k\geq 3$
one can take $g_k(n) = c\sqrt{\frac{\ln n}{\ln \ln n}}$ with some constant $c=c(k)$ by \cite{dm2012}.\\
We show that there exists a constant $c'=c'(k)$ such that for $r\geq n_0+1$, 
$$P_r \geq  c'(rg(r))^2,$$
which then implies the lower bounds in Theorems \ref{ManyColoursTriangle} and \ref{ManyColoursGeneral}. 
 
We prove this statement by induction on $r$. For $r=n_0+1$ this is true provided 
$c'$ is chosen small enough. 
For $r> n_0+1$, by Proposition~\ref{cor:LowerBddWithErdoesRogers} and since 
$f_{k,k+1}$ is  non-decreasing, we have that
\begin{align*}
P_r & \geq P_{r-1} + f_{k,k+1}(P_{r-1})\geq P_{r-1} + \sqrt{P_{r-1}}g(P_{r-1}). 
\end{align*}
Using the induction hypothesis, Lemma~\ref{simpleBdd} and that $g$ is non-decreasing for $r-1\geq n_0$, we obtain
\begin{align*}
P_r & \geq c'((r-1)g(r-1))^2 + \sqrt{c'} (r-1) g(r-1) g(r-1) \\
 & \geq c' (rg(r))^2 + r(g(r-1))^2 \left( \sqrt{c'} -2c' - c'r
 \left( \frac{(g(r))^2}{(g(r-1))^2} -1\right) - \frac{\sqrt{c'}}{r} \right).  
\end{align*}
By our assumption on $g$ the last term is positive, provided $c'$ is small enough.
%
\qed

\section{Packing $(n,r,k)$-critical graphs}\label{sec:packings}

In this section we prove the upper bounds in Theorems~\ref{ManyColoursTriangle}, \ref{ManyColoursGeneral} and \ref{ManyColoursGeneralWeak}. 
Our task is to derive upper bounds for $P_r(k)$, that is we want to find $K_{k+1}$-free
colour patterns such that every $r$-colouring of the vertices produces a strongly monochromatic $K_k$. 
Let us first motivate the idea behind our proofs. Given a colour pattern $G_1,\ldots,G_r$ on an $n$-element vertex set $V$ and any $[r]$-colouring of $V$, at least one of the colours, say $i$, occurs $n/r$ times. If every set of at least $n/r$ vertices in $G_i$ contains a $K_k$, then we must have a strongly monochromatic clique in colour $i$. This motivates the following 
definition: we  call a graph $F$ on $n$ vertices {\em $(n,r,k)$-critical} if 
$K_{k+1}\not\subseteq F$ and $\alpha_{k}(F)<n/r$. We have thus obtained the following lemma.
\begin{lemma} \label{lemma:critical}
If there exists a colour pattern $G_1,\ldots,G_r$ where each $G_i$ is $(n,r,k)$-critical, then $P_r(k) \leq n$.
\end{lemma}
For the rest of this section, we will focus on packing $r$ edge-disjoint 
$(n,r,k)$-critical  graphs into the same $n$-element vertex set, 
such that $n$ is as small as possible. 

In order to produce at least one $(n,r,k)$-critical graph, 
let us recall the Erd\H{o}s-Rogers function,
defined as $f_{k,k+1}(n)=\min\{\alpha_{k}(F)\}$, where the 
minimum is taken over all $K_{k+1}$-free graphs $F$ on $n$ vertices. 
By definition, we have for all $u\in \R$ that 
\begin{align}\label{eq:equivalenceER-NRKcritical}
f_{k,k+1}(n)<u 
&\quad \Longleftrightarrow \quad 
\text{there exists an $(n,n/u,k)$-critical graph. }
\end{align}
So the question whether at least one $(n,r,k)$-critical graph exists on $n$ vertices 
is equivalent to the question whether $f_{k,k+1}(n)<n/r$.


When $k=2$, an $(n,r,2)$-critical graph 
is precisely an $n$-vertex triangle-free graph with 
independence number less than
$n/r$. Hence an $(n,r,2)$-critical graph exists if and only if 
$n < R\left(3,\lceil n/r\rceil\right)$. 
It is known that $R(3,k)=\Theta \left(k^2/\ln k\right)$ where the upper 
bound was first shown by Ajtai, Koml\'os and Szemer\'edi~\cite{aks1980} and the matching 
lower bound was first established by Kim~\cite{k1995}.
Therefore, if $G$ is an $(n,r,2)$-critical graph, then 
$n\geq c\cdot r^2 \ln r$ for some constant $c>0$, 
and $(n,r,2)$-critical graphs do exist for  
$n = C\cdot r^2 \ln r$ for some constant $C>0$.
For our purpose, however, we need to pack $r$ many
$(n,r,2)$-critical graphs in an edge-disjoint fashion into $n$ vertices. 
The next lemma states that we can do so at the 
expense of a factor of $\ln r$. 
\begin{lemma} \label{ManyColoursGadget3}
Let $r$ be an integer. 
Then there exists 
a colour pattern 
$G_1,\ldots,G_r$ on vertex set $[n]$, where $n= O(r^2 \ln^2 r)$, 
such that each $G_i$ is 
$(n,r,2)$-critical.
\end{lemma}
Lemma~\ref{ManyColoursGadget3} together with Lemma~\ref{lemma:critical} and 
Theorem~\ref{fundamental} complete
the proof of Theorem~\ref{ManyColoursTriangle}.

\noindent
For fixed $k\geq 3$, Dudek, Retter, and R\"odl \cite{drr2013}
recently showed that $f_{k,k+1}(n)= O \left((\ln n)^{4k^2} \sqrt{n}\right)$.  
That is, they constructed a
$K_{k+1}$-free graph $F$ on $n$ vertices (where $n$ is large enough) 
such that every subset of $c(\ln n)^{4k^2} \sqrt{n}$ 
vertices contains a $K_k$. 
This is an $(n,r,k)$-critical graph $F$ with $n= c^2\big((2+o(1))\ln r \big)^{8k^2} r^2$. 
Again, we would like to pack $r$ of those graphs into $K_n$. 
But rather than taking a fixed  $(n,r,k)$-critical graph $F$ and pack it into $K_n$, 
we construct $r$ (edge-disjoint) $(n,r,k)$-critical graphs $G_1,\ldots,G_r$ 
simultaneously as subgraphs of $K_n$. 
As it turns out, this simultaneous construction is only little harder than 
the construction itself in \cite{drr2013}; we prove it by black-boxing theorems from \cite{drr2013}.
\begin{lemma} \label{ManyColoursGadget2}
For all integers $k \geq 3$ there exist a constant $C=C(k)>0$ 
and $r_0\in N$ such that, for all $r\geq r_0$, the following holds. 
There exists 
a colour pattern 
$G_1,\ldots,G_r$ on vertex set $[n]$, where $n\leq C \left(\ln r\right)^{8k^2} r^2$, 
such that each $G_i$ is $(n,r,k)$-critical.
\end{lemma} 
Lemma~\ref{ManyColoursGadget2} together with Lemma~\ref{lemma:critical} and 
Theorem~\ref{fundamental} complete
the proof of Theorem~\ref{ManyColoursGeneral}.

For the upper bound in Theorem \ref{ManyColoursGeneralWeak}, 
we are motivated by graphs constructed by Dudek and R\"odl in \cite{dr2011}. 
The graph $F$ on $n$ vertices constructed in \cite{dr2011} is $(n,r,k)$-critical 
with $n=O (k^6r^3)$.
Here it is not as clear to just refer to lemmas from \cite{dr2011} 
in order to do a ``simultaneous'' construction. 
So we will start the construction from scratch and provide all the details needed.  

\begin{lemma} \label{ManyColoursGadget1}
Let $k, r \geq 3$. 
Then there exists 
a colour pattern 
$G_1,\ldots,G_r$ on vertex set $[n]$, where $n\leq 8k^6 r^3$, 
such that each $G_i$ is $(n,r,k)$-critical.
\end{lemma}
Lemma~\ref{ManyColoursGadget1} together with Lemma~\ref{lemma:critical} and 
Theorem~\ref{fundamental} imply Theorem~\ref{ManyColoursGeneral}.

\subsection{Proofs of the Lemmas}
In the rest of this section we prove Lemmas~\ref{ManyColoursGadget3},
~\ref{ManyColoursGadget2}, and~\ref{ManyColoursGadget1}, 
each concerned with packing (edge-disjointly) $r$ graphs $G_1,\ldots,G_r$ 
which are all $(n,r,k)$-critical.  

\medskip
\noindent
{\bf {Packing many $K_3$-free graphs with small independence number.\\}}
Here, we prove Lemma \ref{ManyColoursGadget3}.
To that end, we will show the existence of a graph $F$ on $n:=Cr^2\ln^2r$ vertices, where $C=1000$, 
which can be written as a union of edge-disjoint graphs $G_1$, \ldots, $G_r$ 
which are all $K_3$-free and without independent sets of size $n/r$. 
We will find the graphs $G_i$ successively as subgraphs of $K_n$ 
using the following.
\begin{lemma}[{Lov\'asz Local Lemma, see, e.g., \cite[Lemma 5.1.1]{as2008}}] 
\label{l:lovasz}
Let $A_1,A_2,\ldots,A_n$ be events in an arbitrary probability space. 
A directed graph $D=(V,E)$ on the set of vertices $V=\{1,\ldots,n\}$ is called a 
dependency digraph for the events $A_1,A_2,\ldots,A_n$ if for each $i$, $1\leq i \leq n$, 
the event $A_i$ is mutually independent of all the events $\{A_j : (i,j)\not\in E\}$. 
Suppose that $D=(V,E)$ is a dependency digraph for the above events and 
suppose there are real numbers $x_1,\ldots,x_n$ such that $0\leq x_i< 1$ and 
$\Pr(A_i)\leq x_i \prod_{(i,j)\in E} (1-x_j)$ for all $1\leq i \leq n$. 
Then 
$$\Pr\left( \bigwedge_{i=1}^n \widebar{A_i}\right) \geq \prod_{i=1}^n (1-x_i).$$ 
In particular, with positive probability no event $A_i$ holds. 
\end{lemma}
Given $r$, set $m:= n/r = C r \ln^2r$ and $q := {m \choose 2}/(2r)$. 
For a graph $H$ on $n$ vertices, 
we define $e_{\min}(m,H)$ ($e_{\max}(m,H)$) to be the smallest (largest) 
number of edges that appear in any subset $S \subseteq V(H)$
of size $\size{S} = m$. 
The following lemma is the crucial step to find the graphs $G_i$.

\begin{lemma}\label{lem:LLL}
Let $H = (V,E)$ be a graph on $n$ vertices, where $n\geq n_0$ is large enough, 
and assume $e_{\min}(m,H) \geq {m \choose 2}/2$. 
Then there is a subgraph $H' \subseteq H$ on the same vertex set such that $H'=(V,E')$ is triangle-free, has no independent set on $m$ vertices, and $e_{\max}(m,H') \leq q$.
\end{lemma}

\begin{proof}
Let $c_1=1/4$ and $c_2=1/20$.
Choose $H'$ by including each edge of $H$ independently with probability 
$p:=c_1n^{-1/2}$.
For a subset $S \subseteq V$, let $e(S)$ and $e'(S)$ denote the number 
of edges in $H[S]$ and $H'[S]$, respectively.  
It suffices to show that $H'$ is triangle-free, $e_{\min}(m,H') \geq 1$, and 
$e_{\max}(m,H')\leq q$ with positive probability. 
To that end, we want to apply the 
Lov\'asz Local Lemma, and, therefore, 
we define the set of bad events in the natural way. 
Namely, for every $S \in \binom{V}{3}$ that forms a triangle in $H$, 
we set $T_S$ to be the event that $H'[S]$ is a triangle as well. 
Clearly, the probability of such an event is $p_T := p^3$. 
Further, for every $S \in \binom{V}{m}$, we set $I_S$ to be the event 
that either $S$ is an independent set in $H'$ or satisfies $e'(S) > q$. 
Then,
\begin{align*}
\PP (I_S) &\leq \PP(e'(S) = 0) + \PP (e'(S) \geq q)\\
	&\leq (1-p)^{e(S)} + \binom{e(S)}{q} p^q \\
	&\leq  (1-p)^{\binom{m}{2}/2} + \left(\frac{\binom{m}{2}ep}{q}\right)^q \\
	&= (1-p)^{\binom{m}{2}/2} +  (2epr)^q.
\end{align*}
Note that $(1-p)^{\binom{m}{2}/2} = \exp \left[- p \binom{m}{2}/2\,(1+o(1)) \right]= e^{-pqr(1+o(1))}$ and 
$ (2epr)^q = o(e^{-pqr(1+o(1))})$, 
since $pr \rightarrow 0$, so that for $n$ large enough 
$$ \PP (I_S) \leq 2 (1-p)^{\binom{m}{2}/2} =: p_I.$$
Let $\cE$ be the collection of bad events. That is, 
$\cE = \{T_S : H[S] \cong K_3\} \cup \{I_S : S\in \binom{V}{m} \}$. 
In the auxiliary dependency graph $D$, we connect two of the events 
$A_S,A_{S'} \in \cE$ if $\size{S \cap S'} \geq 2$. 
Then $A_S \in \cE$ is mutually independent from the family of all $A_{S'}$ 
for which $\{A_S,A_{S'}\}$ is not an edge in this dependency graph. 
To apply the Lov\'asz Local Lemma, we now bound the degrees in $D$. We denote by 
$N(E)$ the neighbours in the dependency graph $D$ of the event $E$. 
If $\size{S} = 3$ we have 
\begin{align*}
\size{N(T_S) \cap \{T_{S'} : \size{S'} = 3\}} &\leq 3n, \qquad \text{and}\\
\size{N(T_S) \cap \{I_{S'} : \size{S'} = m\}} &\leq {n \choose m}.\\
\shortintertext{If $\size{S} = m$ we have}
\size{N(I_S) \cap \{T_{S'} : \size{S'} = 3\}} &\leq {m \choose 2}(n-2) < {m \choose 2}n, \qquad \text{and}\\
\size{N(I_S) \cap \{I_{S'} : \size{S'} = m\}} &\leq {n \choose m}.
\end{align*}
Therefore, by Lemma \ref{l:lovasz}, 
if there exist real numbers $x,y \in [0,1)$ such that
\begin{align}
p_T &\leq x(1-x)^{3n}(1-y)^{{n \choose m}}\label{inequ1}\\
p_I &\leq y(1-x)^{{m \choose 2} n}(1-y)^{{n \choose m}}\label{inequ2},
\end{align}
then there exists a graph $H'$ such that none of the events in \cE\ occurs.
We show that these two conditions are fulfilled for 
$x = c_2n^{-3/2}$ and $y={n \choose m}^{-1}$. 
First note that, for $n$ large enough,
$$x(1-x)^{3n}(1-y)^{\binom{n}{m}} = c_2n^{-3/2}e^{-1}(1+o(1)) > p^3,$$ 
so Inequality \eqref{inequ1} holds. 
Now, \eqref{inequ2} is equivalent to 
\begin{align*}
2^{2/\binom{m}{2}}(1-p) 
	&\leq y^{2/\binom{m}{2}} (1-x)^{2n}(1-y)^{2\binom{n}{m}/\binom{m}{2}}.
\end{align*}
We use $1-p\leq e^{-p}$ and $1-z \geq e^{-z-z^2}$ for $z \leq 0.6$ to claim 
\eqref{inequ2} holds if 
\begin{align*}
\exp \left[ \frac{2 \ln 2}{\binom{m}{2}} -p\right] 
	&\leq \exp \left[ \frac{2 \ln y}{\binom{m}{2}} - 2n(x+x^2) - \frac{2 \binom{n}{m}}{\binom{m}{2}} (y+y^2)\right].
\end{align*}
Now, 
$ \frac{2 \ln y}{\binom{m}{2}} \geq - \frac{4}{\sqrt{C}} n^{-1/2} (1+o(1))$
and $1/m^2 = o\left(n^{-1/2}\right)$. 
So \eqref{inequ2} holds if 
\[ \exp \left[ -c_1n^{-1/2} (1+o(1))\right] 
	\leq \exp \left[ - (4/\sqrt{C} +2c_2) n^{-1/2} (1+o(1))\right],\]
which is satisfied by choice of $C,c_1,c_2$.
Applying Lemma \ref{l:lovasz} yields the existence of a subgraph $H'$ such that 
none of the events in \cE\ hold, i.e.\ $H'$ has the desired properties. 
\end{proof}

\begin{proof}[Proof of Lemma \ref{ManyColoursGadget3}]
Let $r$ large enough be given, and set $m:=n/r = C r \ln^2r$ and $q := {m \choose 2}/(2r)$ as before. 
Define $H_1 := K_n$. 
We choose our graphs inductively as subgraphs of $H_1$; 
given $H_i$ for $i \leq r$ such that $e_{\min}(m,H_i) \geq {m \choose 2} - (i-1)q$, we have since $i \leq r$ that 
$$e_{\min}(m,H_i) > {m \choose 2} - rq = \frac{1}{2}{m \choose 2},$$ 
so, by Lemma \ref{lem:LLL}, we may find $G_{i}$ a subgraph of $H_i$ 
with $e_{\max}(m,G_{i}) \leq q$ such that $G_{i}$ is triangle-free 
and has no independent set on $n/r$ vertices. 
Then take $H_{i+1} = H_i - G_{i}$. The graph $H_{i+1}$ will be edge-disjoint 
from $G_{i}$ (and, inductively, from $G_1,\ldots,G_{i-1}$), and 
$$e_{\min}(m,H_{i+1}) \geq e_{\min}(m,H_i) - e_{\max}(m,G_{i}) \geq {m \choose 2} - (i-1)q - q = {m \choose 2} - iq,$$
as desired.
\end{proof}

\medskip
\noindent
{\bf An upper bound tight up to a polylogarithmic factor in $r$\\}
Here, we prove Lemma \ref{ManyColoursGadget2}.
We will rely heavily on the graphs constructed in \cite{drr2013} and use its 
construction as a black box. 

\begin{proof}[Proof of Lemma \ref{ManyColoursGadget2}]
Fix $k \geq 3$ and let $r$ be large enough.
We need to construct $r$ graphs on $n=O(r^2\, (\log r)^{8k^2})$ vertices that 
are $K_{k+1}$-free, but every subset of size $n/r$ contains a $K_k$. 
Let $q$ be the largest prime power such that 
$$ q \leq 128 k (2\log r )^{4k^2} r.$$ 
Then by Bertrand's postulate, $q\geq 64 k(2\log r)^{4k^2}r$, 
and therefore, $q \geq 64 k (\log q)^{4k^2} r$ since $r$ is large enough compared to $k$. 
Consider the affine plane of order $q$. 
It has $n:= q^2$ points and $q^2 +q$ lines such that any two points lie on a unique line, 
every line contains $q$ points, 
and every point lies on $q+1$ lines. 
It is a well-known fact that affine planes exist whenever $q$ is a prime power. 
We call two lines $L$ and $L'$ in the affine plane {\em parallel} 
if $L\cap L'= \emptyset$. 
In the affine plane of order $q$, there exist $q+1$ sets of $q$ 
pairwise disjoint lines. 
Let $(V,\cL)$ be a hypergraph where the vertex set $V$ is the point set of the affine 
plane of order $q$, and the hyperedges are lines of the affine plane, with 
one set of parallel lines removed. 
Then $(V,\cL)$ is a $q$-uniform hypergraph on $q^2$ vertices such that any two hyperedges 
meet in at most one vertex. 

In \cite{drr2013}, Dudek et al.~consider a random subhypergraph $(V,\cL')$ of 
$(V,\cL)$ and show that they can embed the required graph $G$ ``along the hyperedges'' 
of $(V,\cL')$. 
For our purposes, let us call a hypergraph $(V,\cH)$ {\em good} if there exists 
a graph $G$ on vertex set $V$ such that 
\begin{itemize}
\item[(i)] $K_{k+1}\not\subseteq G$, 
\item[(ii)] every subset of size $64k (\log q)^{4k^2} q$ of $V$ contains a $K_k$ in $G$, and 
\item[(iii)] any edge of $G$ lies inside a hyperedge of $\cH$, i.e.\ for every 
$e\in E(G)$ there is some $h\in \cH$ such that $e\subseteq h$. 
\end{itemize}
Clearly, by (i) and (ii) 
any such graph $G$ is $(n,r,k)$-critical, since $\frac{n}{r} = \frac{q^2}{r} > 
64k (\log q)^{4k^2} q$ by the choice of $n$ and $q$.
Though it is not explicitely stated as a lemma, the following is proven in 
Lemma 2.2 of \cite{drr2013}. 
\begin{lemma}[\cite{drr2013} Lemma $2.2^{\ast}$]\label{lem:drrGraph}
Let $(V,\cL')$ be the (random) hypergraph obtained by picking each hyperedge of 
$(V,\cL)$ with probability $\frac{\log^2q}{q}$. Then $(V,\cL')$ is {\em good} with 
probability at least $1/2-o(1)$. 
\end{lemma}
To complete the proof of the lemma 
it would be enough to find $r$ hypergraphs $\cL_1,\ldots,\cL_r$ 
which are {\em good} and satisfy that the hyperedges of different hypergraphs 
intersect in at most one vertex. 
To see this, let $G_i$ be the graph associated with hypergraph $\cL_i$. 
Then, as mentioned above, all the 
graphs $G_i$ are $(n,r,k)$-critical. 
Furthermore, they are edge-disjoint, 
since for every $i$ the edges of $G_i$ lie inside hyperedges of $\cL_i$ by (iii), 
and hyperedges of $\cL_i$ and $\cL_j$ intersect in at most one vertex (since they 
correspond to lines in the affine plane).  

To find the $r$ hypergraphs $\cL_1,\ldots,\cL_r$ 
which are {\em good}, choose a $c$-edge-colouring of $(V,\cL)$ at random, 
where $c := \frac{q}{\log^2 q}$. Note that, since $k\geq 3$ and by choice of $q$, 
$c$ satisfies $c > 4r$. 
Let $\cL_i$ be the sub-hypergraph in colour $i$ ($1\leq i \leq c$). 
Clearly, no two hypergraphs $\cL_i$ and $\cL_j$ contain the same hyperedge. 
Moreover, since hyperedges are lines in the 
affine plane, no two hyperedges intersect in more than one vertex. 
The probability that a line $\ell \in \cL$ is in $\cL_i$ is $\frac{\log^2q}{q}$. 
So $\cL_i$ has the same distribution as the random hypergraph $(V,\cL')$ in 
Lemma \ref{lem:drrGraph}. Therefore, $\cL_i$ is {\em good} with probability at least 1/4, 
provided $q$ is large enough. 
Hence, the expected number of {\em good} hypergraphs $\cL_i$ is 
at least $c/4> r$. 
So, there exists a $c$-colouring of $(V,\cL)$ such that at least $r$ of the monochromatic hypergraphs 
are {\em good}. After relabelling, we have the desired hypergraphs, finishing the proof of Lemma \ref{ManyColoursGadget2}.
\end{proof}

\medskip
\noindent
{\bf An upper bound polynomial in both $k$ and $r$\\}
Here, we prove Lemma \ref{ManyColoursGadget1}.
Let $r\geq 2, k\geq 3$. For $n \leq 8 k^6r^3$ we need to construct $r$
$(n,r,k)$-critical graphs $G_i$ on $n$ vertices
which are edge-disjoint. 
We will define incidence structures $\I_i = (\cP,\cL_i)$ on the same set of points 
such that the families of lines $\cL_i$ are disjoint for distinct $i$. 
Further, any three lines within one $\cL_i$ do not form a triangle. 
We will then, analogously to Dudek and R\"odl \cite{dr2011}, 
enrich the lines in $\cL_i$ randomly, 
and show that the resulting graphs are edge-disjoint and each of them are
$(n,r,k)$-critical with positive probability.  

\begin{proof}[Proof of Lemma \ref{ManyColoursGadget1}]
First, let us define the incidence structures $\I$. 
Let $q$ be the smallest prime power such that $k^2r \leq q$, 
and let $\F_q$ be the finite field of order $q$. 
The common vertex set of our graphs is $V := \F_q^3$, 
i.e. $n=|V| \leq 8k^6r^3$. 
For every $\lambda \in \F_q \setminus \{0\}$, we will define an incidence structure 
$\I_{\lambda} = (V,\cL_{\lambda} )$ where $\cL_{\lambda}$ is a family of lines in $ \F_q^3$. 
For $\lambda \in \F_q \setminus \{0\}$ set 
\[ M_{\lambda} := \Big\{ (1,\lambda \alpha, \lambda \alpha ^2) : \alpha \in \F_q \setminus \{0\} \Big\}. \]
We call $M_{\lambda}$ the  $\lambda$-moment curve. 
In \cite{w1991}, Wenger used the usual moment 
curve $M_1$ to construct dense $C_6$-free graphs. 
Note that for non-zero $\lambda_1 \neq \lambda_2$  
the two curves $M_{\lambda_1}$ and $M_{\lambda_2}$ do not intersect. 
An important and crucial property is that, for any $\lambda \neq 0$,  
any three vectors from $M_{\lambda}$ are linearly independent, 
that is for distinct $\alpha_1, \alpha_2, \alpha_3$,
\[ \det \left( 
	 \begin{array}{ccc}
		1 & \lambda \alpha_1 & \lambda \alpha_1^2 \\
		1 & \lambda \alpha_2 & \lambda \alpha_2^2 \\
		1 & \lambda \alpha_3 & \lambda \alpha_3^2 \end{array} \right)
	= \lambda^2 (\alpha_3-\alpha_1)(\alpha_3-\alpha_2)(\alpha_2-\alpha_1) \neq 0.\]

In general, a line in $\F_q^3$ is a set of the form $\ell_{\bs,\bv} = \{ \beta \bs + \bv  :  \beta \in \F_q \}$, 
where $\bs \in \F_q^3\setminus \{0\}$ is called the {\em slope}. 
We define 
\[ \cL_{\lambda} := \{ \ell_{\bs,\bv} : \bs \in M_{\lambda}, \bv \in \F_q^3\}; \]
that is, in the incidence structure $\I_{\lambda} = (\F_q^3, \cL_{\lambda})$ we only 
allow lines with slope vectors from the $\lambda$-moment curve. 
Clearly, $|\cL_{\lambda}| = \size{M_{\lambda}} \frac{q^3}{q} = q^2 (q-1)$ since each 
line contains $q$ points.
We establish the following properties about each structure $\I_{\lambda}$, $\lambda \neq 0$. 
\begin{itemize}
\item[(1)] Every point $v \in V$ is contained in $q-1$ lines from $\cL_{\lambda}$ and
	every line $\ell \in  \cL_{\lambda}$ contains $q$ points.
\item[(2)] Any two points lie in at most one line. 
\item[(3)] No three lines in $ \cL_{\lambda}$ intersect pairwise in three distinct points 
	(i.e.~form a triangle). 
\end{itemize}
Further, we have for $\lambda_1 \neq \lambda_2$,
\begin{itemize}
\item[(4)] $\cL_{\lambda_1} \cap \cL_{\lambda_2} = \emptyset$.
\end{itemize}
For $(1)$, note that every slope vector in $M_{\lambda}$ gives rise 
to exactly one line through a given point $v \in V$. 
The second part of $(1)$ follows from the definition of a line. 
Property $(2)$ holds because lines are affine subspaces of dimension 1 in the 
vector space $\F_q^3$. 
For $(3)$, suppose three lines in $\cL_{\lambda}$ intersect pairwise in 
three distinct points. 
Then their three slope vectors would be linearly dependent, a contradiction 
to the linear independence of any three vectors in $\cL_{\lambda}$ we established 
above. 
Property $(4)$ simply follows from $M_{\lambda_1} \cap M_{\lambda_2} = \emptyset$ for 
$\lambda_1 \neq \lambda_2$. 

Now, we are ready to define our graphs $G_1,\ldots,G_{q-1}$. 
Let $\lambda \in \F_q \setminus \{0\}$. 
We partition every line $\ell \in \cL_{\lambda}$ randomly into 
$k$ sets $L_1^{(\ell)}$, \ldots, $L_k^{(\ell)}$ each of cardinality 
$l_1:=\left\lfloor \frac{q}{k} \right\rfloor$ 
or $l_2:=\left\lceil \frac{q}{k} \right\rceil$. 
Note that  $l_1, l_2 \geq rk$.
To be precise, between all partitions of a line $\ell = \dot\bigcup_{j=1}^{k} L_j^{(\ell)}$ where 
$\size{L_1^{(\ell)}}=\dots =\size{L_{k'}^{(\ell)}} = l_1$ and 
$\size{L_{k'+1}^{(\ell)}}=\dots =\size{L_{k}^{(\ell)}}=l_2$ 
we choose one uniformly at random, choices for distinct lines in
$\cL_{\lambda}$ being independent.
The graph $G_{\lambda}$ on the vertex set $V = \F_q^3$ is defined as follows. 
For every $\ell \in \cL_{\lambda}$ and any $i \neq j$, 
we include the edges of a complete bipartite graph between the vertex sets $L_i^{(\ell)}$ and $L_j^{(\ell)}$ on $\ell$. 
That is, the graph $G_{\lambda}$ consists of a collection of Tur\'an graphs on $q$ vertices 
with $k$ parts. Each Tur\'an part ``lives'' along one of the lines $\ell \in \cL_{\lambda}$. 
By Property $(2)$, these parts are edge-disjoint. 
Further, by Property $(3)$, $G_{\lambda}$ is $K_{k+1}$-free. 
Also, for distinct $\lambda \in \F_q^3$, by Property $(4)$, the graphs 
$G_{\lambda}$ are edge disjoint. 
To finish the proof, we show that for any fixed 
$\lambda\in \F_q\setminus \{ 0\}$ the graph 
$G_{\lambda}$ is $(n,r,k)$-critical with positive probability. As the choices of the $G_{\lambda}$ are done independently, there is a choice of $G_1,\ldots,G_{q-1}$ with the desired properties.

The calculations are similar to those in \cite{dr2011}. 
For a subset $W \subseteq V(G)$, let $\cA(W)$ denote the event that $G_{\lambda}[W]$ 
contains no $K_k$. 
Let $U\subseteq V(G)$ be a subset of size $|U|= \left\lfloor \frac{n}{r} \right\rfloor$. 
Then, since by Property $(3)$ any $K_k$ can only appear within a line $\ell \in \cL_{\lambda}$, 
\begin{align*}
\cA(U) &\subseteq \bigcap_{\ell \in \cL_{\lambda}} \cA(U \cap \ell), \\
\shortintertext{and therefore, since all the events $\cA(U \cap \ell)$ are independent,}
\PP( \cA(U)) &\leq  \prod_{\ell \in \cL_{\lambda}} \PP(\cA(U \cap \ell)).
\end{align*}
For a line $\ell \in \cL_{\lambda}$, 
set $u_{\ell} := |U \cap \ell|$, and 
let $\ell = \bigcup_{j=1}^{k} L_j^{(\ell)}$ be the partition we chose at random. 
Then the event $\cA(U \cap \ell)$ is equivalent to the existence of a 
$j \in [k]$ such that $U \cap L_j^{(\ell)} = \emptyset$. 
But, for fixed $j \in [k]$, 
\begin{align*} 
\PP \left(U \cap L_j^{(\ell)} = \emptyset \right) &= \frac{\binom{q-u_{\ell}}{\size{L_j^{(\ell)}}}}{\binom{q}{\size{L_j^{(\ell)}}}}
	\leq \left( 1-\frac{u_{\ell}}{q}\right)^{\size{L_j^{(\ell)}}} 
	\leq \exp \left( -\frac{l_1 u_{\ell}}{q}\right).
\end{align*}
Therefore, 
\begin{align*}
\PP( \cA(U)) &\leq  \prod_{\ell \in \cL_{\lambda}} \PP \left(\exists \, j \in [k] : U \cap L_j^{(\ell)} = \emptyset \right)\\
	&\leq k^{|\cL_{\lambda}|} \exp \left( - \sum_{\ell \in \cL_{\lambda}} \frac{l_1 u_{\ell}}{q}\right)\\
	&= k^{|\cL_{\lambda}|} \exp \left( - \frac{q-1}{q} l_1 |U| \right),
\end{align*}
since every point in $U$ belongs to exactly $q-1$ lines (Property $(1)$), and therefore 
$\sum_{\ell \in \cL_{\lambda}} u_{\ell} = \sum_{\ell \in \cL_{\lambda}} |U \cap \ell|= (q-1)|U|$.
We obtain, 
\begin{align*}
\PP \left(\exists U \in \binom{V}{\left\lfloor \frac{n}{r} \right\rfloor} : \cA(U) \right) 
	&\leq \binom{n}{\left\lfloor \frac{n}{r} \right\rfloor}\, k^{|\cL_{\lambda}|}\, 
	\exp \left(-\frac{q-1}{q} l_1 \, \left\lfloor \frac{n}{r} \right\rfloor \right)\\
	&\leq (re)^{n/r}\, k^{q^2(q-1)} \,
		\exp \left(-\frac{q-1}{q} (rk) \left\lfloor \frac{n}{r} \right\rfloor \right)\\
	&\leq \exp \left[q^3 \left( \frac{\ln r}{r} + \frac{1}{r} + \ln k - \frac{3}{4}k\right) \right]\\
	&< 1  
\end{align*}
for $k\geq 3$ and $r\geq 3$. 
Therefore, there exists an instance of $G_{\lambda}$ such that 
every subset $U$ of size at least $\left\lfloor \frac{n}{r} \right\rfloor$ contains a 
$K_k$ in $G_{\lambda}$.
\end{proof}

\section{Concluding remarks}
We have seen, as a consequence of Proposition~\ref{cor:LowerBddWithErdoesRogers} and
Theorem~\ref{fundamental}, 
that $s_r(K_k)\geq s_{r-1}(K_k)$. 
However,  it is not that clear that $s_r(K_k)$ is also increasing in $k$. 
We usually expect that graphs which are Ramsey for $K_k$ should be ``larger''
than those which are Ramsey only for $K_{k-1}$. 
It would be quite unintuitive if the following conjecture was not true.

\begin{conjecture}\label{increasing}
For all $r\geq 3$, $k\geq 3$ we have that $s_r(K_k)\geq s_r(K_{k-1})$.
\end{conjecture}

We also saw that the Erd\H{o}s-Rogers function
is tightly connected to the study of $s_r(K_k)$. 
For our lower bounds in Section~\ref{sec:lowerbdds}, 
we essentially showed that $P_r(k)= \Omega \left(r(f_{k,k+1}(r))^2\right)$, 
provided $g_k(n)=\frac{f_{k,k+1}(n)}{\sqrt{n}}$ is any decent polylogarithmic 
function (which we believe it is).
On the other hand, we saw in Section~\ref{sec:packings} that the known constructions 
for $K_{k+1}$-free graphs with small $k$-independence number can be 
modified to constructions of $r$ pairwise edge-disjoint such graphs on the same or just 
slightly larger vertex set.
In fact, if a packing of essentially optimal $(n,r,k)$-graphs $G$, that is, those with parameters $n/r = \Theta( \alpha_k(G))  = \Theta (f_{k,k+1}(n))$, 
was possible then we would get an 
upper bound that matches our lower bounds. Indeed, then 
$\sqrt{n} = \Theta (r\cdot g_k(n)) = \Theta (r \cdot g_k(r)) = \Theta ( \sqrt{r} \cdot f_{k,k+1}(r))$, so by Lemma~\ref{lemma:critical} we would have
$$P_r(k)\leq n =  \Theta\left(r(f_{k,k+1}(r))^2\right).$$
We strongly believe the following is true.

\begin{conjecture} For every fixed $k\geq 3$,
$$s_r(K_{k}) = \Theta \left( r\cdot \left( f_{k-1,k}(r)\right)^2\right).$$
\end{conjecture}

Therefore, we believe that tightening the known bounds on $f_{k-1,k}(n)$ 
will directly contribute to tightening the bounds on $s_r(K_k)$. 
The currently best known bounds~\cite{drr2013} on the Erd\H{o}s-Rogers function are 
$$ \Omega \left(\sqrt{\frac{n \ln n}{\ln\ln n}}\right)
	= f_{k,k+1}(n)
	= O\left(  (\ln n)^{4k^2} \sqrt{n} \right),$$ 
so it is not yet clear how strongly the logarithmic factor depends on $k$.
We wonder whether the upper bound can be strengthened in the following way. 
\begin{question}
Does there exist a universal constant $C$ (independent of $k$) 
such that 
$f_{k,k+1}(n) = O\left((\ln n)^{C}\sqrt{n}\right)$? 
And does the construction of such a $K_{k+1}$-free graph on $n$ vertices 
with $k$-independence number less than $O\left((\ln n)^{C}\sqrt{n} \right)$ 
generalise to a packing of such graphs?
\end{question}
A positive answer to both questions would imply that there is a universal 
constant $C>0$ such that $s_r(K_k)= O(r^2 (\ln r)^{C})$.

In the special case of $K_3$, in the proof of Lemma~\ref{ManyColoursGadget3}  
we iteratively applied the  Local Lemma to 
find edge-disjoint triangle-free subgraphs $G_i\subseteq K_n$ 
with independence number  $O(\sqrt{n}\ln n)$
and this implied our upper bound in Theorem \ref{ManyColoursTriangle}. 
This approach was an adaptation of the classical application of the Local Lemma
by Spencer~\cite{s1977} to lower bound off-diagonal Ramsey numbers and
obtain $R(3,k)\geq c \left(k/\ln k\right)^2$.
Subsequently Kim~\cite{k1995} proved the existence of a triangle-free graph $G$ on $n$ 
vertices with independence number $O \left(\sqrt{n \ln n}\right)$, hence
establishing that correct order of magnitude of $R(3,k)$ is $k^2/\ln k$. 
Earlier Bollob\'as and Erd\H{o}s suggested an alternative approach to the problem
of finding better lower bounds on $R(3,k)$: the triangle-free process.
In 2009, Bohman \cite{b2009} managed to reprove Kim's theorem using the triangle-free 
process. Very recently, 
Fiz Pontiveros, Griffiths and Morris \cite{fgm2013}, and independently Bohman and Keevash \cite{bk2013}, improved the constant factor in the analysis
and showed that $R(3,k) \geq (1/4 - o(1)) k^2 / \ln k$. 
We are optimistic that one can apply the triangle-free process iteratively, 
with some modifications, and thus find not only one, but
a packing of triangle-free graphs $G_1,\ldots,G_r$ on $n$ vertices, all  
having independence number $O(\sqrt{n\ln n})$. Thus, we conjecture that 
our lower bound on $s_r(K_3)$ is tight.
\begin{conjecture}
$s_r(K_3) = \Theta \left(r^2 \ln r\right)$. 
\end{conjecture}

\noindent {\bf Acknowledgements.} We thank Tom Bohman for helpful discussions. 

\bibliographystyle{abbrv}
\bibliography{referencesAll}

\end{document}